\documentclass[12pt]{amsart}
\usepackage[centertags]{amsmath}
\usepackage{amsfonts}
\usepackage{amssymb}
\usepackage[latin5]{inputenc}
\usepackage{amsthm}
\usepackage{mathrsfs}
\usepackage{upgreek}
\usepackage{amsopn}
\usepackage{amscd}
\newtheorem{theorem}{Theorem}[section]

\theoremstyle{definition}

\theoremstyle{remark}

\numberwithin{equation}{section}

\begin{document}
{\setlength{\baselineskip}%
                        {1.3\baselineskip}
                      
\title{A note on the Paper ``Homomorphisms with respect to a function"}
\author{Z. Ercan}
\address{Department of Mathematics, Abant Izzet Baysal University, Golkoy Kampusu, Bolu}
\email{zercan@ibu.edu.tr}


\maketitle 
\begin{abstract} We give a direct proof of the main result of the paper ``Homomorphisms with respect to a function" by K. Boulabiar and F. Gdara \cite{bg} without using the Axiom of Choice.
\end{abstract}
\section{The Theorem}
\noindent For a topological space $X$, the lattice-ordered ring (under the pointwise algebraic operations and pointwise ordering) of all real-valued continuous functions on $X$ is denoted by $C(X)$. Recall that  a completely regular Hausdorff space $X$ is called {\it realcompact} if it is homeomorphic to a closed subspace of the product space of the reals. For details about the  lattice ordered group $C(X)$ and the notion of realcompactness, we refer to \cite{gj}.

For $r\in\mathbb{R}$, the map ${\bf r}:\mathbb{R}\rightarrow\mathbb{R}$ is defined by ${\bf r}(x):=r$. For a nonempty index set $I$ and $j\in I$, the mapping $P_{j}:\prod_{i\in I}\mathbb{R}\rightarrow\mathbb{R}$ is defined by $$P_{j}((x_{i})_{i\in I}):=x_{j},$$ and $e_{j}:=(x_{i})_{i\in I}$ with $x_{j}=1$ and $x_{i}=0$ for $i\not =j$. We note that $P_{j}\in C(\prod_{i\in I}\mathbb{R})$.

For a topological space $X$, $H:C(X)\rightarrow\mathbb{R}$ denotes a positive group homomorphism with 
$H({\bf 1})=1$. We note that, in this case, $H$ is linear. For a ${\varphi}\in C(\mathbb{R})$, the map $H$ is called a {\it ${\varphi}$-homomorphism} if  
 $$H\circ {\varphi}={\varphi}\circ H,$$ that is, for every $f\in C(X)$, we have
$$H({\varphi}\circ f)={\varphi}(H(f)).$$ We say that $H$ is {\it point evaluated} if there exists $k\in X$ such that
$$H(f)=f(k)$$ for all $f\in C(X)$. 

The following is the main result of \cite{bg}.
\begin{theorem}\label{teo:1} Let $X$ be a realcompact space. The following are equivalent:
\begin{enumerate}
\item[\textnormal{(i)}] $H$ is point  evaluated at some point of $X$.
\item[\textnormal{(ii)}] $H$ is a $\varphi$-homomorphism for all ${\varphi}\in C(\mathbb{R})$ with ${\varphi}(r) > {\varphi}(0)$ for all $r\in\mathbb{R}\setminus\{0\}$.
\item[\textnormal{(iii)}] There exists ${\varphi}\in C(R)$ with ${\varphi}(r) > {\varphi}(0)$ for all $r\in\mathbb{R}\setminus\{0\}$ such that H is a
$\varphi$-homomorphism.
\end{enumerate}
\end{theorem}
\section{Proof Of Theorem \ref{teo:1}}
\noindent In \cite{bg}, using the notions of Stone-\v{C}ech compactification and Stone-extensions, the proof of the above theorem is given, whereby the Axiom of choice has implicitly been used. Following \cite{eo}, we can give a direct proof of the above theorem without using the Axiom of Choice.

\begin{proof}[Proof of Theorem \ref{teo:1}.] In the proof we mainly follow the arguments of \cite{eo}, so to be as self contained as possible, we repeat some parts of the proof therein.  As indicated in \cite{bg}, the implications $\text{(i)}\Longrightarrow \text{(ii)}$ and $\text{(ii)}\Longrightarrow \text{(iii)}$ are obvious. To see $\text{(iii)}\Longrightarrow \text{(i)}$, suppose that for a ${\varphi}\in C(\mathbb{R})$, the map $H$ is a  $\varphi$-homomorphism. Without loss of generality, we can suppose that ${\varphi}(0)=0$. There are two cases:

{\it Case 1}. $X=\prod_{i\in I}\mathbb{R}$.

Let $c:=(c_{i}):= ({\varphi}(P_{i}))$. We show that for each $0\leq f\in C(X)$, one has 
${\varphi}(f) = f(c)$.
For each $f\in C(X)$, let
$k_{f} : X\rightarrow \mathbb{R}$ be defined by $k_{f}:=f-f(c)\, {\bf 1}$.
Then ${\varphi}(k_{f}) = {\varphi}(f)-f(c)$ and
$H(k_{f}) = 0 $ if and only if $H(f) = f(c)$.
Hence it is enough to show that
$f(c) = 0$ implies $H(f) = 0$. Let $0\leq f\in C(X)$ be given. 

{\it Claim $(\dagger)$}: if $f|_{U}=0$  for some open set $U$ which contains $c$, then $H(f) = 0$.
There exists a family $(U_{i})_{i\in I}$ of open subsets of of $\mathbb{R}$ such that
$$c\in \prod_{i\in I}U_{i} = V\subset U\quad\mbox{and}\quad F = \{i\in I : U_{i}\not =\mathbb{R}\}\quad\mbox{is finite}.$$
Let $h:X\rightarrow\mathbb{R}$ be defined by 
$$h:=\sum_{i\in F}{\varphi}\circ (P_{i}-{\bf c_{i}}).$$ Note that
$$H(h)=0.$$
It is clear that $h(x)\not =0$ whenever $x\not\in V$. Define $g:X\rightarrow\mathbb{R}$ by 
$$g(x):={{f(x)}\over {h(x)}}{\chi}_{X\setminus V}.$$ Then $g$ is continuous and $f=gh$. For each $n\in\mathbb{N}$, since
$$0\leq g-g\wedge {\bf n}\leq {1\over n}\, g^{2},$$ we have 
$$0\leq gh-(g\wedge {\bf n})h=gh-(gh)\wedge {\bf n}h\leq {1\over n}\, g^{2}h.$$ Notice also that for each $n\in\mathbb{N}$, since $H$ is positive, we have
$$H((g\wedge {\bf n})h)=H(gh\wedge nh)\leq H(nh)=nH(h)=0.$$ This implies that
$$0\leq H(f)=H(fg)=H(fg-gh\wedge nh)\leq {1\over n}(g^{2}h)$$ for each $n$, whence $H(f)=0$. 

{\it Claim $(\ddagger)$}. if $f(c)=0$, then $H(f)=0$. Suppose that it is not. Then we may suppose that $H(f)=1$ and $f(c)=0$. For each $n\in\mathbb{N}$, let 
$$U_{n}=f^{-1}(-{1\over n},{1\over n}).$$ Then $c\in U_{n}$ and 
$$(f-f\wedge {1\over n})|_{U_{n}}=0.$$ By Claim $(\ddagger)$, for all $n$, we have 
$$1=H(f)=H(f\wedge {1\over n})\leq H({1\over n})={1\over n}\rightarrow 0,$$ which is a contradiction. Clearly this shows that $H(f) = 0$ whenever $f(c) =0$.

{\it Case 2}. $X$ is homeomorphic to a closed subset $Y=\prod_{i\in I}\mathbb{R}$.

Let ${\pi}:C(Y)\rightarrow C(X)$ be defined by 
$${\pi}(f):=f|_{X},$$ where $f|_{X}$ denotes the restriction of $f$ to $X$. Then there exists $c\in Y$ such that
$H\circ {\pi}(f) = f(c)$ for each $f\in C(Y)$. Suppose that $c\not\in X$. As $Y$ is a completely regular Hausdorff space, there exists
$f\in  C(Y)$ such that
$f(c) = 1$ and $f|_{X}=0$. This implies that
$$1=f(c) =P\circ {\pi}(f) = P(f|X) = 0,$$
which is a contradiction. Hence $H(f)=f(c)$ for all $0\leq f\in C(X)$. For arbitrary $f\in C(X)$, we then have 
$$H(f)=H(f^{+}-f^{-})=H(f^{+})-H(f^{-})=f^{+}(c)-f^{-}(c)=f(c).$$ This completes the proof. 
\end{proof}
In \cite{e} it is noticed that if $X$ is a completely regular Hausdorff space then a nonzero linear map,  $H:C(X)\rightarrow \mathbb{R}$ with $H({\bf 1})=1$, is a ring homomorphism if and only if there exists a net $(x_{\alpha})$ in $Y$ such that 
$$H(f)=\lim f(x_{\alpha})$$ for all $f\in C(X)$. By combining this with the above Theorem immediately we get the following theorem. 
\begin{theorem} Let $X$ be a completely regular Hausdorff space and $H:C(X)\rightarrow\mathbb{R}$ be a positive group homomorphism with $H({\bf 1})=1$. Then the following are equivalent.
\\
i.) $H$ is a Ring homomorphism.
\\
ii.) $H$ is a $\varphi$-homomorphism for all ${\varphi}\in C(\mathbb{R})$ satisfying ${\varphi}(x)>{\varphi}(0)$ for all $x\not =0$.
\\
iii.) $H$ is a Riesz homomorphism.
\end{theorem}

\end{document}